\documentclass[11pt]{amsart}
\usepackage[english]{babel}
\usepackage{amssymb,enumerate,bbm}

\catcode`\<=\active \def<{
\fontencoding{T1}\selectfont\symbol{60}\fontencoding{\encodingdefault}}
\catcode`\>=\active \def>{
\fontencoding{T1}\selectfont\symbol{62}\fontencoding{\encodingdefault}}
\catcode`\|=\active \def|{
\fontencoding{T1}\selectfont\symbol{124}\fontencoding{\encodingdefault}}
\newcommand{\assign}{:=}
\newcommand{\dueto}[1]{\textup{\textbf{(#1) }}}
\newcommand{\mathd}{\mathrm{d}}
\newcommand{\nocomma}{}
\newcommand{\tmem}[1]{{\em #1\/}}

\newcommand{\tmmathbf}[1]{\ensuremath{\boldsymbol{#1}}}
\newcommand{\tmname}[1]{\textsc{#1}}
\newcommand{\tmop}[1]{\ensuremath{\operatorname{#1}}}

\newcommand{\tmtextit}[1]{{\itshape{#1}}}
\newenvironment{enumeratenumeric}{\begin{enumerate}[1.] }{\end{enumerate}}
\newtheorem{lemma}{Lemma}
\newtheorem{theorem}{Theorem}
\newtheorem*{theorem*}{Theorem}

\renewcommand{\paragraph}[1]{ {\mbox{}\\ \noindent\textbf{#1.}} }

\begin{document}

\title[Second order Riesz transforms and processes with jumps]{Second order Riesz transforms on multiply--connected Lie groups and
processes with jumps}

\author{N.\;Arcozzi}

\author{K.\;Domelevo}

\author{S.\;Petermichl}

\date{\today}

{\maketitle}

\begin{abstract}
  We study a class of combinations of second order Riesz transforms on Lie
  groups $\mathbbm{G}=\mathbbm{G}_{x} \times \mathbbm{G}_{y}$ that are
  multiply connected, composed of a discrete abelian component
  $\mathbbm{G}_{x}$ and a compact connected component $\mathbbm{G}_{y}$. We
  prove sharp $L^{p}$ estimates for these operators, therefore generalizing
  previous results
  {\cite{VolNaz2004a}}{\cite{DomPet2014c}}{\cite{BanBau2013}}.
  
  We construct stochastic integrals with jump components adapted to functions
  defined on the semi-discrete set $\mathbbm{G}_{x} \times \mathbbm{G}_{y}$.
  We show that these second order Riesz transforms applied to a function may
  be written as conditional expectation of a simple transformation of a
  stochastic integral associated with the function. The analysis shows that
  It{\^o} integrals for the discrete component must be written in an
  {\tmem{augmented}} discrete tangent plane of dimension twice larger than
  expected, and in a suitably chosen discrete coordinate system. Those
  artifacts are related to the difficulties that arise due to the discrete
  component, where derivatives of functions are no longer local. Previous
  representations of Riesz transforms through stochastic integrals in this
  direction do not consider discrete components and jump processes. 
\end{abstract}

\section{Introduction}

Sharp $L^{p}$ inequalities for pairs of differentially subordinate martingales
date back to the celebrated work of {\tmname{Burkholder}} {\cite{Bur1984a}} in
1984 where the optimal constant is exhibited. See also from the same author
{\cite{Bur1988a}}{\cite{Bur1991a}}. The relation between differentially
subordinate martingales and Cald{\'e}ron-Zygmund operators is known since
{\tmname{Gundy}}--{\tmname{Varopoulos}} {\cite{GunVar1979}}.
{\tmname{Banuelos}}--{\tmname{Wang}} {\cite{BanWan1995}} were the first to
exploit this connection to prove new sharp inequalities for singular
intergrals. A vast literature has since then been accumulating on this line of
research, some of which will be discussed below.

\medskip

In this article we bring for the first time this whole circle of ideas to a
semi-discrete setting, applying it to a family of second order Riesz
transforms on multiply--connected Lie groups.  We prove optimal norm 
estimates in $L^p$ for these operators as well as derive their representation 
through stochastic integrals using jump processes on multiply--connected Lie groups.

\medskip

Before we state our results in a
complete form, we present a sample case which requires little notation.

Consider a Lie group $\mathbbm{G} \assign \mathbbm{G}_{x} \times
\mathbbm{G}_{y} ,$ where $\mathbbm{G}_{x}$ is a discrete abelian group with a
fixed finite set $G= \{ g_{i} , \nocomma g_{i}^{-1} : i=1, \ldots ,m \}$ of
generators, and their reciprocals, and $\mathbbm{G}_{y}$ is a connected,
compact Lie group of dimension $n$ endowed with a biinvariant metric. Given a
smooth function $f:\mathbbm{G} \rightarrow \mathbbm{C}$ consider a fixed
orthonormal basis $Y_{1} , \ldots ,Y_{n}$ of the Lie algebra
$\mathfrak{G}_{y}$ of $\mathbbm{G}_{y}$. The gradient of a function and the
divergence of a vector field with respect to the variable $y$, and the Laplace
operator $\Delta_{y}$, can all be expressed in terms of the fixed basis.
Discrete partial derivatives are defined in analogy. Given any point $( x,y )
\in \mathbbm{G}_{x} \times \mathbbm{G}_{y}$, and given a direction $i \in \{
1, \ldots ,m \}$, the right and the left derivative at $( x,y )$ in the
direction $i$ are:
\[ \partial_{i}^{+} f ( x,y ) \assign f ( x+ g_{i} ,y \nocomma ) -f ( x ,y )
   =:X_{i}^{+} f ( x,y ) , \]
\[ \partial^{-}_{i} f ( x,y ) \assign f ( x ,y ) -f ( x-g_{i}  ,y )
   =:X_{i}^{-} f ( x,y ) . \]
The discrete Laplace operator $\Delta_{x}$ is then defined as $$\Delta_{x} f (
x,y ) = \sum_{i=1}^{m} X_{i}^{-} X_{i}^{+} f ( x,y ).$$ Let $z= ( x,y )$ and
$\Delta_{z} = \Delta_{x} + \Delta_{y}$. For $1<p< \infty$, let $p^{\ast} =
\max \left\{ p, \frac{p}{p-1} \right\}$. The $L^{p}$-norm of a function is
computed w.r.t. to the Haar measure on $\mathbbm{G}$.

The inequalities in the Theorem below can be rephrased in terms of the Riesz
transform estimates proved in Theorem \ref{T: p minus 1 estimate}.

\begin{theorem*}
If $f:\mathbbm{G} \rightarrow \mathbbm{C}$ is smooth and $1<p<
\infty$, then
\[ \| Y_{j} Y_{l} f \|_{L^{p}} \leqslant ( p^{\ast} -1 ) \| ( - \Delta_{z} ) f
   \|_{L^{p}} \]
and
\[ \| X^{+}_{i} X^{-}_{i} f \|_{L^{p}} \leqslant ( p^{\ast} -1 ) \| ( -
   \Delta_{z} ) f \|_{L^{p}} \]
for $1 \leqslant j,l \leqslant n$ and $1 \leqslant i \leqslant m$.
\end{theorem*}

The inequalities in the Theorem can be rephrased in terms of Riesz transforms. 
These can be formally defined as $R_i^2:=R_i^+\circ R_i^-:=(X^{+}_{i}\circ (-\Delta_z)^{-1/2})\circ (X^{-}_{i}\circ (-\Delta_z)^{-1/2})=
X^{+}_{i}\circ X^{+}_{i}\circ (-\Delta_z)^{-1})$, $1\le i\le m$, in the discrete directions,
and $R_{jk}=R_j\circ R_k=...=Y_j\circ Y_k\circ (-\Delta_z)^{-1}$, $1\le j,k\le n$ in the continuous directions.

The standard procedure for obtaining inequalities for singular integrals from
inequalities for martingales can be described as follows. Starting with a test
function $f$, martingales are built using Brownian motion and Poisson
extensions in the upper half space $\mathbbm{R}^{+} \times \mathbbm{R}^{n}$.
It is shown that the martingale arising from $R f$, where $R$ is a Riesz
transform in $\mathbbm{R}^{n}$, is a martingale transform of that arising from
$f$. The two form a pair of martingales with differential subordination and
orthogonality. Thus, in the case of Riesz tranforms, the optimal $L^{p}$
constants could be recovered using probabilistic methods. One derives
martingale inequalities under hypotheses of strong differential subordination
and orthogonality relations, see {\tmname{Banuelos}}--{\tmname{Wang}}
{\cite{BanWan1995}}.

In the case of second order Riesz transforms, the use of heat extensions in
the upper half space instead of Poisson extensions originated in
{\tmname{Petermichl}}-{\tmname{Volberg}} {\cite{PetVol2002a}} and was used to
prove \ $L^{p}$ estimates for the second order Riesz transforms based on the
results of Burkholder in {\tmname{Nazarov}}--{\tmname{Volberg}}
{\cite{VolNaz2004a}} as part of their best-at-time estimate for the
Beurling-Ahlfors operator, whose real and imaginary parts themselves are
second order Riesz transforms. 

Here is how this idea takes shape in our case.

\begin{theorem*}
  The second order Riesz transforms $R_i^2f$, $1\le i\le m$, and $R_{jk}f$, $1\le j,k\le n$,
 of a function $f \in L^{2} ( \mathbbm{G} )$ as defined in
  (\ref{eq: definition 2nd order Riesz transforms}) can be written as the
  conditional expectations 
  $$R_i^2f(z)=\mathbbm{E} ( M_{0}^{i,f} | \mathcal{Z}_{0}=z ) \text{ and } R_{jk}f(z)=\mathbbm{E} ( M_{0}^{j,k,f} | \mathcal{Z}_{0}=z ).$$ 
  Here $M_{t}^{i,f}$ and $M_{t}^{j,k,f}$ are suitable martingale transforms of the
  martingale $M^{f}_{t}$ associated to $f$, and $\mathcal{Z}_{t}$ is a
  suitable random walk on $\mathbbm{G}$ (see Section \ref{S: stochastic
  integrals and martingale transforms}).
\end{theorem*}

All these $L^{p}$ norm inequalities use special functions found in
{\tmname{Pichorides}} {\cite{Pic1972}}, {\tmname{Ess{\'e}n}} {\cite{Ess1984}},
{\tmname{Banuelos}}--{\tmname{Wang}} {\cite{BanWan1995}} when orthogonality is
present in addition to differential subordination or {\tmname{Burkholder}}
{\cite{Bur1984a}}{\cite{Bur1987a}}{\cite{Bur1988a}}, {\tmname{Wang}}
{\cite{Wan1995a}} when differential subordination is the only hypothesis.

Deterministic proofs of sharp $L^{p}$ estimates of Cald{\'e}ron-Zygmund
operators that use Burkholder's theorems are available in the literature. The
technique of Bellman functions was used in
{\tmname{Nazarov}}--{\tmname{Volberg}} {\cite{VolNaz2004a}} for an $L^{p}$
estimate for certain second order Riesz transforms in the Euclidean plane as
well as in the recent version on discrete abelian groups
{\tmname{Domelevo--Petermichl}} {\cite{DomPet2014c}}.

\medskip

The aim of the present work is two--fold. On the one hand, we want to
generalize the estimate to second order Riesz transforms acting on multiply
connected Lie groups, built as the cartesian product of a discrete abelian
group with a connected compact Lie group. Previous works based on stochastic
methods for the analysis of Riesz transforms on connected compact Lie groups
are in {\tmname{Arcozzi}} {\cite{Arc1995a}}{\cite{Arc1998a}}, and sharp
$L^{p}$ estimates were proved in this setting in
{\tmname{Ba{\~n}uelos}}--{\tmname{Baudoin}} {\cite{BanBau2013}} for second
order Riesz transforms. The novelty of this text is the generalization to the
multiply connected setting. In this sense, it is also a generalization of
{\cite{DomPet2014c}}, by regarding each point in the discrete abelain group as
a Lie group of dimension zero.

On the other hand we want to derive a probabilistic proof through the use of
an identity formula involving stochastic integration. We make use of
stochastic integrals with jump components adapted to functions defined on the
semi-discrete set $\mathbbm{G}_{x} \times \mathbbm{G}_{y}$. We show that
corresponding second order Riesz transforms applied to a function may be
written as conditional expectation of a simple transformation of a stochastic
integral associated with the function. The desired $L^{p}$ estimates then
follow from the work by {\tmname{Wang}} {\cite{Wan1995a}}, who had identified
the correct requirements on differential subordination in the presence of
discontinuity in space.

\medskip

Such probabilistic representations can be very advantageous. For example, 
the conjectured estimate for the $L^{p}$ norms
of the Beurling-Ahlfors transform is a fascinating, famous open question. The
argument in {\cite{VolNaz2004a}} is deterministic in spirit and gave a new
best estimate at its time, twice larger than expected. Subsequent improvements
profited substantially from the finer structure that remains intact through
the use of stochastic integration and a probabilistic representation formula
of second order Riesz transforms due to
{\tmname{Ba{\~n}uelos--M{\'e}ndez--Hernandez}}{\tmname{}}
{\cite{BanMen2003a}}. This representation was subsequently used on many
occasions, notably in the improvement for the norm estimate of the
Beurling-Ahlfors operator by {\tmname{Ba{\~n}uelos}}--{\tmname{Janakiraman}}
{\cite{BanJan2008a}}. Their result was a \ major advance in this direction,
since it was the first that dropped the constant below twice larger than
exected, thus confirming the suspicion that the conjectured estimate should
hold true. Further improvements were made by
{\tmname{Borichev--Janakiraman--Volberg}} {\cite{BorJanVol2013b}}, also
through the refined use of stochastic integration formulae and a deep, new
martingale estimate.

\medskip

The sharpness of the constant for the real part of the Beurling-Alfhors
operator, a combination of perfect squares of Riesz transforms, was proved
using probabilistic methods in conjunction with a modification of a technique
by Bourgain in {\tmname{Geiss}}--{\tmname{Montgomery}}--{\tmname{Saksman}}
{\cite{GeiMonSak2010a}}. It turns out the real part of the Beurling-Ahlfors
operator alone already attains the conjectured norm estimate. See also
applications in {\tmname{Ba{\~n}uelos}}--{\tmname{Baudoin}}
{\cite{BanBau2013}}.

\medskip

We expect our representation of semi--discrete second order Riesz transforms
to have further applications, such as to UMD spaces (see
{\cite{GeiMonSak2010a}}), as well as logarithmic and weak--type estimates (see
{\cite{Ose2013a}}{\cite{Ose2014a}}), currently under investigation with
{\tmname{Osekowski}}.

\medskip

{\tmname{Banuelos}}--{\tmname{Wang}} {\cite{BanWan1995}} adressed
c{\`a}dl{\`a}g processes in the presence of orthogonality relations, such as
one sees in the continuous setting when considering martingales that represent
the Hilbert transform through stochastic integrals. It is remarkable that the
$L^{p}$ estimates of the discrete Hilbert transform on the integers are still
unknown. It is a famous conjecture that this operator has the same norm as its
continuous counterpart. We hope advances made in this paper give new ideas on how to address this question.

\medskip

The authors thank Adam Osekowsi for interesting insight to the subject and
pointing out some of the applications.

\subsection{Differential operators and Riesz transforms}

\mbox{}\smallskip

\paragraph{First order derivatives and tangent planes}
We will consider Lie groups $\mathbbm{G} \assign \mathbbm{G}_{x} \times
\mathbbm{G}_{y} ,$ where $\mathbbm{G}_{x}$ is a discrete abelian group with a
fixed set $G$ of $m$ generators, and their reciprocals, and $\mathbbm{G}_{y}$
is a connected, compact Lie group of dimension $n$ endowed with a biinvariant
metric. The choice of the set $G$ of generators in $\mathbbm{G}_{x}$
corresponds to the choice of a bi-invariant metric structure on
$\mathbbm{G}_{x}$. We will use on $\mathbbm{G}_{x}$ the multiplicative
notation for the group operation. We will define a product metric structure on
$\mathbbm{G}$, which agrees with the Riemannian structure on the first factor,
and with the discrete ``word distance'' on the second. We will at the same
time define a ``tangent space'' \ $T_{z} \mathbbm{G}$ for $\mathbbm{G}$ at a
point $z= ( x,y ) \in ( \mathbbm{G}_{x} \times \mathbbm{G}_{y} )
=\mathbbm{G}$. We will do this in three steps.

First, since $\mathbbm{G}_{y}$ is an $n$-dimensional connected Lie group with
Lie algebra $\mathfrak{G}_{y}$. We can identify each left-invariant vector
field $Y$ in $\mathfrak{G}_{y}$ with its value at the identity $e$,
$\mathfrak{G}_{y} \equiv T_{e} \mathbbm{G}_{y}$. Since $\mathbbm{G}$ is
compact, it admits a bi-invariant Riemannian metric, which is unique up to a
multiplicative factor. We normalize it so that the measure $\mu_{y}$
associated with the metric satisfies $\mu_{y} ( \mathbbm{G}_{y} ) =1$. The
measure $\mu_{y}$ is also the normalized Haar measure of the group. We denote
by $< \cdot , \cdot >_{y}$ be the corresponding inner product on $T_{y}
\mathbbm{G}_{y}$ and by $\tmmathbf{\nabla_{y}} f ( y )$ the gradient at $y \in
\mathbbm{G}_{y}$ of a smooth function $f:\mathbbm{G}_{y} \rightarrow
\mathbbm{R}$. Let $Y_{1} , \ldots ,Y_{n}$ be a orthonormal basis for
$\mathfrak{G}_{y}$. The gradient of $f$ can be written $\tmmathbf{\nabla_{y}}
f=Y_{1} ( f ) Y_{1} + \ldots +Y_{n} ( f ) Y_{n}$.

Second, in the discrete component $\mathbbm{G}_{x}$, let $\mathfrak{G}_{x} = (
g_{i} )_{i=1, \ldots ,m}$ be a set of generators for $\mathbbm{G}_{x}$, such
that for $i \neq j$ and $\sigma = \pm 1$ we have $g_{i} \neq g_{j}^{\sigma}$.
The choice of a particular set of generators induces a word metric, hence, a
geometry, on $\mathbbm{G}_{x}$. Any two sets of generators induce bi-Lipschitz
equivalent metrics.

At any point $x \in \mathbbm{G}_{x}$, and given a direction $i \in \{ 1,
\ldots ,m \}$, we can define the right and the left derivative at $x$ in the
direction $i$:
\[ ( \partial^{+} f/ \partial x_{i} ) ( x,y ) \assign f ( x+ g_{i} ,y ) -f ( x
   ,y ) \assign ( \partial_{i}^{+} f ) ( x,y ) \]
\[ \hspace{1em} ( \partial^{-} f/ \partial x_{i} ) ( x,y ) \assign f ( x ,y )
   -f ( x-g_{i}  ,y ) \assign ( \partial_{i}^{-} f ) ( x,y ) . \]
Comparing with the continuous component, this suggests that the tangent plane
$\hat{T}_{x} \mathbbm{G}_{x}$ at a point $x$ of the discrete group
$\mathbbm{G}_{x}$ might actually be split into a ``right'' tangent plane
$T_{x}^{+} \mathbbm{G}_{x}$ and a ``left'' tangent plane $T_{x}^{-}
\mathbbm{G}_{x}$, according to the direction with respect to which discrete
differences are computed. We consequently define the {\bf{augmented}}
discrete gradient $\widehat{\tmmathbf{\nabla}}_{x} f ( x )$, with a
{\tmem{hat}}, as the $2m$--vector of $\hat{T}_{x} \mathbbm{G}_{x} \assign
T_{x}^{+} \mathbbm{G}_{x} \oplus T_{x}^{-} \mathbbm{G}_{x}$ accounting for all
the local variations of the function $f$ in the direct vicinity of $x$; that
is, the $2m$--column--vector
\[ \widehat{\tmmathbf{\nabla}}_{x} f ( x ) \assign ( X_{1}^{+} f,X_{2}^{+} f,
   \ldots ,X_{1}^{-} f,X_{2}^{-} f, \ldots ) ( x ) = \sum_{i=1}^{m} \sum_{\tau
   = \pm} X_{i}^{\tau} f ( x ) \hspace{1em} \] 
   with $X_{i}^{\tau}  \in \hat{T}_{x} \mathbbm{G}_{x}, $
where we noted the discrete derivatives $X_{i}^{\pm} f \assign
\partial_{i}^{\pm} f$ and introduced the discrete $2m$--vectors
$X_{i}^{\pm}$ as the column vectors of $\mathbbm{Z}^{2m}$
\[ X_{i}^{+} = ( 0, \ldots ,1, \ldots ,0 ) \times \tmmathbf{0}_{m} ,
   \hspace{1em} X_{i}^{-} =\tmmathbf{0}_{m} \times ( 0, \ldots ,1, \ldots ,0 )
   .\]
Here the $1$'s in $X_{i}^{\pm}$ are located at the $i$--th position of
respectively the first or the second $m$--tuple. Notice that those vectors are
independent of the point $x$. The scalar product on $\hat{T}_{x}
\mathbbm{G}_{x} \assign T_{x}^{+} \mathbbm{G}_{x} \oplus T_{x}^{-}
\mathbbm{G}_{x}$ is defined as
\[ ( U,V )_{\hat{T}_{x} \mathbbm{G}_{x}} \assign \frac{1}{2} \sum_{i=1}^{m}
   \sum_{\tau = \pm} U_{i}^{\tau} V^{\tau}_{i}  . \]
We chose to put a factor $\frac{1}{2}$ in front of the scalar product to
compensate for the fact that we consider both left and right differences.

Finally, for a function $f$ defined on the cartesian product $\mathbbm{G}
\assign \mathbbm{G}_{x} \times \mathbbm{G}_{y}$, the (augmented) gradient
$\widehat{\tmmathbf{\nabla}}_{z} f ( z )$ at the point $z= ( x,y )$ is an
element of the tangent plane $\hat{T}_{z} \mathbbm{G} \assign \hat{T}_{x}
\mathbbm{G}_{x} \oplus T_{y} \mathbbm{G}_{y}$, that is a $( 2m+n
)$--column--vector
\begin{eqnarray*}
  \widehat{\tmmathbf{\nabla}}_{z} f ( z ) & \assign & \sum_{i=1}^{m}
  \sum_{\tau = \pm} X_{i}^{\tau} f ( z )  \hat{X}_{i}^{\tau} +
  \sum_{j=1}^{n} Y_{j} f ( z )  \hat{Y}_{j} ( z )\\
  & = & ( X_{1}^{+} f,X_{2}^{+} f, \ldots ,X_{1}^{-} f,X_{2}^{-} f, \ldots
  ,Y_{1} f,Y_{2} f, \ldots ) ( z )
\end{eqnarray*}
where $\hat{X}_{i}^{\tau}$ and $\hat{Y}_{j} ( z )$ can be identified with
column vectors of size $( 2m+n )$ with obvious definitions and scalar product
$( \cdot , \cdot )_{\hat{T}_{z} \mathbbm{G}_{z}}$.

Let $\mathd \mu_{z} \assign \mathd \mu_{x} \mathd \mu_{y}$, $\mathd \mu_{x}$
being the counting measure on $\mathbbm{G}_{x}$ and $\mathd \mu_{y}$ being \
the Haar measure on $\mathbbm{G}_{y}$. The inner product of $\varphi , \psi$
in $L^{2} ( \mathbbm{G} )$ is
\begin{eqnarray*}
  ( \varphi , \psi )_{L^{2} ( \mathbbm{G} )} & \assign & \int_{\mathbbm{G}}
  \varphi ( z ) \psi ( z ) d \mu_{z} ( z ) .
\end{eqnarray*}

\paragraph{Riesz transforms} Following {\cite{Arc1995a}}{\cite{Arc1998a}},
recall first that for a compact Riemannian manifold $\mathbbm{M}$ without
boundary, one denotes by $\tmmathbf{\nabla}_{\mathbbm{M}}$,
$\tmop{div}_{\mathbbm{M}}$ and $\Delta_{\mathbbm{M}} \assign
\tmop{div}_{\mathbbm{M}} \tmmathbf{\nabla}_{\mathbbm{M}}$ respectively the
gradient, the divergence and the Laplacian associated with $\mathbbm{M}$. Then
$- \Delta_{\mathbbm{M}}$ is a positive operator and the vector Riesz transform
is defined as the linear operator
\[ \tmmathbf{R}_{\mathbbm{M}} \assign \tmmathbf{\nabla}_{\mathbbm{M}} \circ (
   - \Delta_{\mathbbm{M}} )^{-1/2} \]
acting on $L^{2}_{0} ( \mathbbm{M} )$ ($L^{2}$ functions with vanishing mean).
It follows that if $f$ is a function defined on $\mathbbm{M}$ and $y \in
\mathbbm{M}$ then $\tmmathbf{R}_{\mathbbm{M}} f ( y )$ is a vector of the
tangent plane $T_{y} \mathbbm{M}$.

Similarily on $\mathbbm{M}=\mathbbm{G}$, we define
$\tmmathbf{\nabla}_{\mathbbm{G}} \assign \widehat{\tmmathbf{\nabla}}_{z}$ as
before, and then we define the divergence operator as its formal adjoint, that
is $- \tmop{div}_{\mathbbm{G}} =- \widehat{  \tmop{div}}_{z} \assign
\widehat{\tmmathbf{\nabla}}_{z}^{\ast}$, with respect to the natural $L^{2}$
inner product of vector fields:
\begin{eqnarray*}
  ( U,V )_{L^{2} ( \hat{T} \mathbbm{G} )} & \assign & \int_{\mathbbm{G}} ( U (
  z ) ,V ( z ) )_{\hat{T}_{z} \mathbbm{G}} \; \mathd \mu_{z} ( z )
\end{eqnarray*}
We have the $L^{2}$-adjoints $( X_{i}^{\pm} )^{\ast} =-X_{i}^{\mp}$ and
$Y_{j}^{\ast} =-Y_{j}$. If $U \in \hat{T} \mathbbm{G}$ is defined by
\[ U ( z ) = \sum_{i=1}^{m} \sum_{\tau = \pm} U_{i^{}}^{\tau} ( z )
   \hat{X}_{i}^{\tau} + \sum_{j=1}^{n} U_{j} ( z ) \hat{Y}_{j} , \]
we define its divergence $\widehat{\tmmathbf{\nabla}}_{z}^{\ast} U$ as
\begin{eqnarray*}
  \widehat{\tmmathbf{\nabla}}_{z}^{\ast} U ( z ) & \assign & - \frac{1}{2}
  \sum_{i=1}^{m} \sum_{\tau = \pm} X_{i}^{- \tau} U_{i}^{\tau} ( z ) -
  \sum_{j=1}^{n} Y_{j} U_{j} ( z ) .
\end{eqnarray*}
The Laplacian $\Delta_{\mathbbm{G}}$ is as one might expect:
\begin{eqnarray*}
  \Delta_{z} f ( z ) & \assign & - \widehat{\tmmathbf{\nabla}}_{z}^{\ast}
  \widehat{\tmmathbf{\nabla}}_{z} f ( z ) =-
  \widehat{\tmmathbf{\nabla}}_{x}^{\ast} \widehat{\tmmathbf{\nabla}}_{x} f ( z
  ) - \widehat{\tmmathbf{\nabla}}^{\ast}_{y} \widehat{\tmmathbf{\nabla}}_{y} f
  ( z )\\
  & = & \sum_{i=1}^{m} X_{i}^{-} X_{i}^{+} f ( z )  +
   \sum_{j=1}^{n} Y^{2}_{j} f ( z )\\
  & = & \sum_{i=1}^{m} X_{i}^{2} f ( z )  + 
  \sum_{j=1}^{n} Y^{2}_{j} f ( z )\\
  & =: & \Delta_{x} f ( z )  +  \Delta_{y} f ( z )
\end{eqnarray*}
where we denoted $X_{i}^{2} \assign X_{i}^{+} X_{i}^{-} =X_{i}^{-} X_{i}^{+}$.
We have chosen signs so that $- \Delta_{\mathbbm{G}} \geqslant 0$ as an
operator. The Riesz vector $( \hat{\tmmathbf{R}}_{z} f ) ( z )$ is the $( 2m+n
)$--column--vector of the tangent plane $\hat{T}_{z} \mathbbm{G}$ defined as
the linear operator
\[ \hat{\tmmathbf{R}}_{z} f \assign \left( \widehat{\tmmathbf{\nabla}}_{z} f
   \right) \circ ( - \Delta_{z} f )^{-1/2} \]
We also define transforms along the coordinate directions:
\[ R^{\pm}_{i} = X^{\pm}_{i} \circ ( - \Delta_{z} )^{-1/2} \hspace{1em}
   \tmop{and} \hspace{1em} R_{j} = Y_{j} \circ ( - \Delta_{z} )^{-1/2}
   . \]
If $\mathbbm{G}_{x}$ is a finite group, then the transforms $R_{j}$ apply to
$$L^{p}_{0} ( \mathbbm{G} ) \assign \left\{ f \in L^{p} ( \mathbbm{G} )  
\tmop{such}   \tmop{that}   \tmop{for}  x \in \mathbbm{G}_{x}   \tmop{one}  
\tmop{has}   \sum_{x \in \mathbbm{G}_{x}} \int_{\mathbbm{G}_{y}} f ( x, \cdot
)  d \mu_{y} =0 \right\} $$

\subsection{Main results}

In this text, we are concerned with second order Riesz transforms and
combinations thereof. We first define the square Riesz transform in the
(discrete) direction $i$ to be
\[ R_{i}^{2} \assign R_{i}^{+} R_{i}^{-} =R_{i}^{-} R_{i}^{+} . \]
Then, given $\alpha \assign ( ( \alpha^{x}_{i} )_{i=1 \ldots m} , (
\alpha^{y}_{j k} )_{j,k=1 \ldots n} ) \in \mathbbm{C}^{m} \times
\mathbbm{C}^{n \times n}$, we define $R_{\alpha}^{2}$ to be the following
combination of second order Riesz transforms:
\begin{equation}
  R_{\alpha}^{2} \assign \sum^{m}_{i=1} \alpha^{x}_{i} \; R^{2}_{i} +
  \sum^{n}_{j,k=1} \alpha^{y}_{j k} \; R_{j} R_{k} , \label{eq: definition 2nd
  order Riesz transforms}
\end{equation}
where the first sum involves squares of discrete Riesz transforms as defined
above, and the second sum involves products of continuous Riesz transforms.
This combination is written in a condensed manner as the quadratic form
\[ R^{2}_{\alpha} = \left( \hat{\tmmathbf{R}}_{z} , \tmmathbf{A}_{\alpha}
   \hat{\tmmathbf{R}}_{z} \right) \]
where $\tmmathbf{A}_{\alpha}$ is the $( 2m+n ) \times ( 2m+n )$ block matrix
\begin{equation}
  \tmmathbf{A}_{\alpha} \assign \left(\begin{array}{cc}
    \tmmathbf{A}_{\alpha}^{x} & \tmmathbf{0}\\
    \tmmathbf{0} & \tmmathbf{A}_{\alpha}^{y}
  \end{array}\right) \label{eq: Matrix Aalpha}
\end{equation}
with
\[ \tmmathbf{A}_{\alpha}^{x} = \tmmathbf{\tmop{diag}} ( \alpha_{1}^{x} ,
   \ldots , \alpha_{m}^{x} , \alpha_{1}^{x} , \ldots , \alpha_{m}^{x} )
   \in \mathbbm{C}^{2m \times 2m}, 
   \tmmathbf{A}_{\alpha}^{y} = ( \alpha^{y}_{j k} )_{j,k=1 \ldots n}
    \in \mathbbm{C}^{n \times n} . \]
The first result is a representation formula of second order Riesz transforms
$R^{2}_{\alpha}$ \tmtextit{{\`a} la} {\tmname{Gundy--Varopoulos}} (see
{\cite{GunVar1979}}).

\begin{theorem}
  \label{T: a la Gundy-Varopoulos}The second order Riesz transform
  $R_{\alpha}^{2} f$ of a function $f \in L^{2} ( \mathbbm{G} )$ as defined in
  (\ref{eq: definition 2nd order Riesz transforms}) can be written as the
  conditional expectation $$\mathbbm{E} ( M_{0}^{\alpha ,f} | \mathcal{Z}_{0}
  =z ).$$ Here $M_{t}^{\alpha ,f}$ is a suitable martingale transform of a
  martingale $M^{f}_{t}$ associated to $f$, and $\mathcal{Z}_{t}$ is a
  suitable random walk on $\mathbbm{G}$ (see Section \ref{S: stochastic
  integrals and martingale transforms}).
\end{theorem}

When $p$ and $q$ are conjugate exponents, let $p^{\ast} = \max \left\{ p,
\frac{p}{p-1} \right\}$. We have the estimate

\begin{theorem}
  \label{T: p minus 1 estimate}Let $\mathbbm{G}$ be a Lie group and
  $R^{2}_{\alpha} :L_{0}^{p} ( \mathbbm{G},\mathbbm{C} ) \rightarrow L^{p} (
  \mathbbm{G},\mathbbm{C} )$ be a combination of second order Riesz transforms
  as defined above. This operator satisfies the estimate
  \[ \| R_{\alpha}^{2} \| \leqslant \left\| \tmmathbf{A}_{\alpha} \right\|_{2}
     \  ( p^{\ast} -1 ) . \]
  The estimate above is sharp when the group $\mathbbm{G}=\mathbbm{G}_{x}
  \times \mathbbm{G}_{y}$ and $\dim ( \mathbbm{G}_{y} ) + \dim^{\infty} (
  \mathbbm{G}_{x} ) \geqslant 2$, where $\dim^{\infty} ( \mathbbm{G}_{x} )$
  denotes the number of infinite components of $\mathbbm{G}_{x}$.
\end{theorem}

Above, we have set: $$\left\| \tmmathbf{A}_{\alpha} \right\|_{2} = \max \left(
\left\| \tmmathbf{A}_{\alpha}^{x} \right\|_{2} , \left\|
\tmmathbf{A}_{\alpha}^{y} \right\|_{2} \right) = \max \left( | \alpha_{1}^{x}
| , \ldots , | \alpha_{m}^{x} | , \left\| \tmmathbf{A}_{\alpha}^{y}
\right\|_{2} \right).$$
In the case where $\mathbbm{G}=\mathbbm{G}_{x}$ only consists of the discrete
component, this whas proved in {\cite{DomPet2014b}}{\cite{DomPet2014c}} using
the deterministic Bellman function technique. In the case where
$\mathbbm{G}=\mathbbm{G}_{y}$ is a connected compact Lie group, this was
proved in {\cite{BanBau2013}} using Brownian motions defined on manifolds and
projections of martingale transforms.

In the case where the function $f$ is real valued, we obtain better estimates
involving the {\tmname{Choi}} constants (see {\tmname{Choi}}
{\cite{Cho1992a}}). Compare with {\cite{BanOse2012a}}{\cite{DomPet2014c}}.

\begin{theorem}
  \label{T: Choi constant estimate}Assume that $a\tmmathbf{I} \leqslant
  \tmmathbf{A}_{\alpha} \leqslant b\tmmathbf{I}$ in the sense of quadratic
  forms, where $a,b$ are real numbers. Then $R^{2}_{\alpha} :L^{p} (
  \mathbbm{G},\mathbbm{R} ) \rightarrow L^{p} ( \mathbbm{G},\mathbbm{R} )$
  enjoys the norm estimate $\| R_{\alpha}^{2} \|_{p} \leqslant 
  \mathfrak{C}_{a, \nocomma b,p} \nocomma$, where these are the Choi
  constants. 
\end{theorem}

The Choi constants (see {\cite{Cho1992a}}) are not explicit, except
$\mathfrak{C}_{-1,1,p} =p^{\ast} -1$. An approximation of
$\mathfrak{C}_{0,1,p}$ is known and writes as
\[ \mathfrak{C}_{0,1,p} = \tfrac{p}{2} + \tfrac{1}{2} \log \left(
   \tfrac{1+e^{-2}}{2} \right) + \tfrac{\beta_{2}}{p} + \ldots ., \]
    with $\beta_{2} = \log^{2} \left( \tfrac{1+e^{-2}}{2}
   \right) + \tfrac{1}{2} \log \left( \tfrac{1+e^{-2}}{2} \right) -2 \left(
   \tfrac{e^{-2}}{1+e^{-2}} \right)^{2} . $

\subsection{Weak formulations}

Let $f:\mathbbm{G} \rightarrow \mathbbm{C}$ be given. The heat extension
$\tilde{f} ( t )$ of $f$ is defined as $\tilde{f} ( t ) \assign e^{t
\Delta_{z}} f=:P_{t} f$. We have therefore $\tilde{f} ( 0 ) =f$. The aim of
this section is to derive weak formulations for second order Riesz transforms.
We start with the weak formulation of the identity operator $\mathcal{I}$.

\begin{lemma}
  \label{L: weak identity formula}Assume $f$ and $g$ in $L_{0}^{2} (
  \mathbbm{G} )$, then
  \begin{eqnarray*}
    ( \mathcal{I} f,g ) & = & ( f,g )_{L^{2} ( \mathbbm{G} )}\\
    & = & 2 \int^{\infty}_{0} \left( \widehat{\tmmathbf{\nabla}}_{z} P_{t} f,
    \widehat{\tmmathbf{\nabla}}_{z} P_{t} g \right)_{L^{2} ( \hat{T}
    \mathbbm{G} )} \; \mathd  t\\
    & = & 2 \int^{\infty}_{0} \int_{z \in \mathbbm{G}} \left\{ \frac{1}{2}
    \sum_{i=1}^{m} \sum_{\tau = \pm} ( X_{i}^{\tau} P_{t} f ) ( z )
     ( X_{i}^{\tau} P_{t} g ) ( z ) \  + \right.\\
    &  & \left. \sum_{j=1}^{n} ( Y_{j} P_{t} f ) ( z ) ( Y_{j}
    P_{t} g ) ( z ) \right\}  \mathd \mu_{z} ( z )   \mathd  t
  \end{eqnarray*}
  and the sums and integrals that arise converge absolutely.
\end{lemma}

\begin{proof}
  This classical formula can be obtained by observing that $\mathd_{t} P_{t} =
  \Delta_{z} P_{t}$ and writing the ODE satisfied by $\phi ( t ) \assign (
  P_{t} f,P_{t} g )_{L^{2} ( \mathbbm{G} )}$.
\end{proof}

In order to pass to the weak formulation for the squares of Riesz transforms,
we need the following hypothesis and commutation properties.

\paragraph{Hypothesis}We assume everywhere in the sequel:
\begin{enumeratenumeric}
  \item The discrete component $\mathbbm{G}_{x}$ of the Lie group
  $\mathbbm{G}$ is an abelian group
  
  \item The connected component $\mathbbm{G}_{y}$ of the Lie group
  $\mathbbm{G}$ is a compact Lie group endowed with a biinvariant Riemannian
  metric, so that the family $( Y_{j} )_{j=1, \ldots ,n}$ commutes with
  $\Delta_{y}$.
\end{enumeratenumeric}
\begin{lemma}
  {\dueto{Commutation relations}}Assuming the Hypothesis above, we have
  \begin{eqnarray*}
    Y_{j} \circ \Delta_{z} & = & \Delta_{z} \circ Y_{j}\\
    X_{i}^{\tau} \circ \Delta_{z} & = & \Delta_{z} \circ X_{i}^{\tau} ,
    \hspace{1em} \tau \in \{ +,- \}
  \end{eqnarray*}
\end{lemma}

\begin{proof}
  Since $\mathbbm{G}=\mathbbm{G}_{x} \times \mathbbm{G}_{y}$ is a cartesian
  product, we have the commutator $[ Y_{j} ,X_{i}^{\tau} ] =0$ and as a
  consequence $[ Y_{j} , \Delta_{x} ] =0$ and $[ X_{i}^{\tau} , \Delta_{y} ]
  =0$. Then $[ Y_{j} , \Delta_{y} ] =0$ yields $[ Y_{j} , \Delta_{z} ] =0$,
  and since $\mathbbm{G}_{x}$ is abelian we have successively $[ X_{i}^{\tau}
  , \Delta_{x} ] =0$ and $[ X_{i}^{\tau} , \Delta_{z} ] =0$.
\end{proof}

\begin{lemma}
  \label{L: representation formula}Assume the Hypothesis and the Commutation
  lemma above. Assume $f$ and $g$ in $L^{2}_{0} ( \mathbbm{G} )$, then
  \begin{eqnarray*}
    ( R^{2}_{\alpha} f,g )_{L^{2} ( \mathbbm{G} )} & = & -2 \int^{\infty}_{0}
    \left( \tmmathbf{A}_{\alpha} \widehat{\tmmathbf{\nabla}}_{z} P_{t} f,
    \widehat{\tmmathbf{\nabla}}_{z} P_{t} g \right)_{L^{2} ( \hat{T}
    \mathbbm{G} )} \; \mathd  t\\
    & = & -2 \int^{\infty}_{0} \int_{z \in \mathbbm{G}} \Bigg \{ \frac{1}{2}
    \sum_{i=1}^{m} \sum_{\tau = \pm} \alpha_{i}^{x} \; ( X_{i}^{\tau} P_{t} f
    ) ( z )  ( X_{i}^{\tau} P_{t} g ) ( z ) \\
    &  & +  \sum_{j,k=1}^{n} \alpha_{j k}^{y} \; ( Y_{j}
    P_{t} f ) ( z ) ( Y_{k} P_{t} g ) ( z ) \vphantom{\int_{z \in \mathbbm{G}}}   \Bigg \}  \; \mathd
    \mu_{z} ( z ) \; \mathd  t
  \end{eqnarray*}
  and the sums and integrals that arise converge absolutely.
\end{lemma}

\begin{proof}
  We apply Lemma (\ref{L: weak identity formula}) to $R^{2}_{\alpha} f$
  instead of $f$ and we are left with integrands of the form
  \begin{eqnarray*}
    \lefteqn{ \left( \widehat{\tmmathbf{\nabla}}_{z} P_{t} R^{2}_{\alpha} f,
    \widehat{\tmmathbf{\nabla}}_{z} P_{t} g \right)_{L^{2} ( \hat{T}
    \mathbbm{G} )}    }\\
    & = & ( ( - \Delta_{z} ) P_{t} R^{2}_{\alpha} f,P_{t} g
    )_{L^{2} ( \mathbbm{G} )}\\
    & = & \sum_{i} \alpha_{i}^{x}  ( ( - \Delta_{z} ) P_{t}
    R^{2}_{i} f,P_{t} g ) + \sum_{j,k} \alpha_{j,k}^{y} ( ( - \Delta_{z} )
    P_{t} R_{j} R_{k} f,P_{t} g )\\
    & = & \sum_{i} \alpha_{i}^{x}  ( ( - \Delta_{z} ) P_{t}
    X_{i}^{-} ( - \Delta_{z} )^{-1/2} X_{i}^{+} ( - \Delta_{z} )^{-1/2}
    f,P_{t} g )\\
    &  & +  \sum_{j,k} \alpha_{j,k}^{y}  ( ( -
    \Delta_{z} ) P_{t} X_{j} ( - \Delta_{z} )^{-1/2} X_{k} ( - \Delta_{z}
    )^{-1/2} f,P_{t} g )\\
    & = & \frac{1}{2} \sum_{i} \sum_{\tau = \pm} \alpha_{i}^{x} 
    ( X_{i}^{\tau} P_{t} f,X_{i}^{\tau} P_{t} g )  + 
    \sum_{j,k} \alpha^{y}_{j,k}  ( X_{j} P_{t} f,X_{k} P_{t} g )\\
    & = & \left( \tmmathbf{A}_{\alpha} \widehat{\tmmathbf{\nabla}}_{z} P_{t}
    f, \widehat{\tmmathbf{\nabla}}_{z} P_{t} g \right)_{L^{2} ( \hat{T}
    \mathbbm{G} )}
  \end{eqnarray*}
  where we used successively the commutation properties of the Laplacian
  $\Delta_{z}$ with the vector fields and the commutation properties of the
  vector fields with $P_{t} =e^{t \Delta_{z}}$. This yields the desired
  result.
\end{proof}

\section{Stochastic integrals and martingale transforms}\label{S: stochastic
integrals and martingale transforms}

In what follows, we assume that we have a complete probability space $( \Omega
, \mathcal{F} ,\mathbbm{P} )$ with a c{\`a}dl{\`a}g (i.e. right continuous
left limit) filtration $( \mathcal{F}_{t} )_{t \geqslant 0}$ of
sub-$\sigma$--algebras of $\mathcal{F}$. We assume as usual that
$\mathcal{F}_{0}$ contains all events of probability zero. All random walks
and martingales are adapted to this filtration.

We define below a continuous-time random process $\mathcal{Z}$ with values in
$\mathbbm{G}$, $\mathcal{Z}_{t} \assign ( \mathcal{X}_{t} , \mathcal{Y}_{t} )
\in \mathbbm{G}_{x} \times \mathbbm{G}_{y}$, having infinitesimal generator
$L= \Delta_{z}$. The pure-jump component $\mathcal{X}_{t}$ is a compound
Poisson jump process on the discrete set $\mathbbm{G}_{x}$, wheras the
continuous component $\mathcal{Y}_{t}$ is a standard brownian motion on the
manifold $\mathbbm{G}_{y}$. Then, It{\^o}'s formula ensures that semi-discrete
``harmonic'' functions $f:\mathbbm{R}^{+} \times \mathbbm{G} \rightarrow
\mathbbm{C}$ solving the backward heat equation $( \partial_{t} + \Delta_{z} )
f=0$ give rise to martingales $M_{t}^{f} \assign f ( t, \mathcal{Z}_{t} )$ for
which we define a class of martingale transforms.

\paragraph{Stochastic integrals on Riemannian manifolds and It{\^o}
integral}Following {\tmname{Emery}} {\cite{Eme2000a}}{\cite{Eme2005a}}, see
also {\tmname{Arcozzi}} {\cite{Arc1995a}}{\cite{Arc1998a}}, we define the
Brownian motion $\mathcal{Y}_{t}$ on $\mathbbm{G}_{y}$, a compact Riemannian
manifold, as the process $\mathcal{Y}_{t} \  : \  \Omega
\rightarrow ( 0,T ) \times \mathbbm{G}_{y}$ such that for all smooth functions
$f:\mathbbm{G}_{y} \rightarrow \mathbbm{R}$, the quantity
\begin{equation}
  f ( \mathcal{Y}_{t} ) -f ( \mathcal{Y}_{0} ) - \frac{1}{2} \int_{0}^{t} (
  \Delta_{y} f ) ( \mathcal{Y}_{s} ) \; \mathd s=: ( I_{\mathd_{y} f} )_{t}
  \label{eq: Ito for continuous processes}
\end{equation}
is an $\mathbbm{R}$--valued continuous martingale. For any adapted continuous
process $\Psi$ with values in the cotangent space $T^{\ast} \mathbbm{G}_{y}$
of $\mathbbm{G}_{y}$, if $\Psi_{t} ( \omega ) \in T^{\ast}_{Y_{t} ( \omega )}
\mathbbm{G}_{y}$ for all $t \geqslant 0$ and $\omega \in \Omega$, then one can
define the {\tmem{continuous}} It{\^o} integral $I_{\Psi}$ of $\Psi$ as
\[ ( I_{\Psi} )_{t} \assign \int_{0}^{t} \langle \Psi_{s} , \mathd
   \mathcal{Y}_{s} \rangle \]
so that in particular
\[ ( I_{\mathd_{y} f} )_{t} \assign \int_{0}^{t} \langle \mathd_{y} f (
   \mathcal{Y}_{s} ) , \mathd \mathcal{Y}_{s} \rangle \]
The integrand above involves the $1$--form of $T_{y}^{\ast} \mathbbm{G}_{y}$
\[ \mathd_{y} f ( y ) \assign \sum_{j} ( Y_{j} f ) ( y ) \ 
   Y_{j}^{\ast} . \]

\paragraph{A pure jump process on $\mathbbm{G}_{x}$}We will now define the
{\tmem{discrete}} $m$--dimensional process $\mathcal{X}_{t}$ on the
{\tmem{discrete}} abelian group $\mathbbm{G}_{x}$ as a generalized compound
Poisson process. In order to do this we need a number of independent variables
and processes:

First, for any given $1 \leqslant i \leqslant m$, let
$\mathcal{N}^{i}_{t}$ be a c{\`a}dl{\`a}g Poisson process of parameter
$\lambda$, that is
\[ \forall t, \hspace{1em} \mathbbm{P} ( \mathcal{N}^{i}_{t} =n ) = \frac{(
   \lambda  t )^{n}}{n!} e^{- \lambda t} . \]
The sequence of instants where the jumps of the $\mathcal{N}^{i}_{t}$ occur is
noted $( T^{i}_{k} )_{k \in \mathbbm{N}}$, with the convention $T^{i}_{0} =0$.

Second, we set
\[ \mathcal{N}_{t} = \sum_{i=1}^{m} \mathcal{N}_{t}^{i} \]
Almost surely, for any two distinct $i$ and $j$, we have $\{ T_{k}^{i} \}_{k
\in \mathbbm{N}} \cap \{ T_{k}^{j} \}_{k \in \mathbbm{N}} = \emptyset$. Let
therefore $\{ T_{k} \}_{k \in \mathbbm{N}} = \cup_{i=1}^{m} \{ T_{k}^{i} \}_{k
\in \mathbbm{N}}$ be the ordered sequence of instants of jumps of
$\mathcal{N}_{t}$ and let $i_{t} \equiv i_{t} ( \omega )$ be the index of the
coordinate where the jump occurs at time $t$. We set $i_{t} =0$ if no jump
occurs. The random variables $i_{t}$ are measurable: $i_{t} = (
\mathcal{N}^{1}_{t} -\mathcal{N}_{t^{-}}^{1} ,\mathcal{N}^{2}_{t}
-\mathcal{N}_{t^{-}}^{2} , \ldots ,\mathcal{N}^{m}_{t}
-\mathcal{N}_{t^{-}}^{m} ) \cdot ( 1,2, \ldots ,m )$. In differential form,
\[ \mathd \mathcal{N}_{t} = \sum_{i=1}^{m} \mathd \mathcal{N}^{i}_{t} = \mathd
   \mathcal{N}^{i_{t}}_{t} . \]

Third, we denote by $( \tau_{k} )_{k \in \mathbbm{N}}$ a sequence of
independent Bernoulli variables
\[ \forall k, \hspace{1em} \mathbbm{P} ( \tau_{k} =1 ) =\mathbbm{P} ( \tau_{k}
   =-1 ) =1/2. \]

Finally, the random walk $\mathcal{X}_{t}$ started at $\mathcal{X}_{0} \in
\mathbbm{G}_{x}$ is the c{\`a}dl{\`a}g compound Poisson process (see e.g.
{\tmname{Protter}} {\cite{Pro2005a}}, {\tmname{Privault}}
{\cite{Pri2009a,Pri2014a}}) defined as
\[ \mathcal{X}_{t} \assign \mathcal{X}_{0} + \sum_{k=1}^{\mathcal{N}_{t}}
   G_{i_{k}}^{\tau_{k}} , \]
where $G_{i}^{\tau} = ( 0, \ldots ,0, \tau g_{i} ,0, \ldots ,0 )$ when $i_{}
\neq 0$ \ and $( 0, \ldots ,0 )$ when $i_{} =0$.

\paragraph{Stochastic integrals on discrete groups}We recall for the
convenience of the reader the derivation of stochastic integrals for jump
processes. We will emphasize the fact that the corresponding It{\^o}'s formula
involves the action of a discrete $1$--form written in a well-chosen local
coordinate system of the discrete {\tmem{augmented}} cotangent plane (see
details below). Let $1 \leqslant k \leqslant \mathcal{N}_{t}$ and let $( T_{k}
,i_{k} , \tau_{k} )$ be respectively the instant, the axis and the direction
of the $k$--th jump. We set $T_{0} =0$. Let $f \assign f ( t,x )$, $t \in
\mathbbm{R}^{+}$, $x \in \mathbbm{G}_{x}$ a function defined on
$\mathbbm{R}^{+} \times \mathbbm{G}_{x}$. Then
\begin{eqnarray*}
  \lefteqn{f ( t, \mathcal{X}_{t} ) - f ( 0,\mathcal{X}_{0} )}\\
  & = & f ( t, \mathcal{X}_{t} ) -f \left(
  T_{\mathcal{N}_{t}} , \mathcal{X}_{T_{\mathcal{N}_{t}}} \right) +
  \sum_{k=1}^{\mathcal{N}_{t}} \{ f ( T_{k} , \mathcal{X}_{T_{k}} ) -f (
  T_{k-1} , \mathcal{X}_{T_{k-1}} ) \} \\
  & = & f ( t, \mathcal{X}_{t} ) -f \left( T_{\mathcal{N}_{t}} ,
  \mathcal{X}_{T_{\mathcal{N}_{t}}} \right)\\
  &  & + \sum_{k=1}^{\mathcal{N}_{t}} \{ f ( T_{k} , \mathcal{X}_{T_{k}} ) -f
  ( T_{k^{}} , \mathcal{X}_{T_{k-}} ) +f ( T_{k} , \mathcal{X}_{T_{k-}} ) -f (
  T_{k-1} , \mathcal{X}_{T_{k-1}} ) \} \\
  & = & \int_{T_{\mathcal{N}_{t}}}^{t} ( \partial_{t} f ) ( s,
  \mathcal{X}_{s} ) \mathd s\\
  &  & + \sum_{k=1}^{\mathcal{N}_{t}} \left\{ f ( T_{k} , \mathcal{X}_{T_{k}}
  ) -f ( T_{k} , \mathcal{X}_{T_{k-}} ) + \int_{T_{k-1}}^{T_{k}} (
  \partial_{t} f ) ( s, \mathcal{X}_{s} ) \mathd s \right\}\\
  & = & \int_{0}^{t} ( \partial_{t} f ) ( s, \mathcal{X}_{s} ) \mathd s+
  \int_{0}^{t} ( f ( s, \mathcal{X}_{s} ) -f ( s, \mathcal{X}_{s_{-}} ) ) \;
  \mathd \mathcal{N}_{s} \\
  & = & \int_{0}^{t} ( \partial_{t} f ) ( s, \mathcal{X}_{s} ) \mathd s+
  \sum_{i=1}^{m} \int_{0}^{t} ( f ( s, \mathcal{X}_{s} ) -f ( s,
  \mathcal{X}_{s_{-}} ) ) \; \mathd \mathcal{N}^{i}_{s} .
\end{eqnarray*}
At an instant $s$, the integrand in the last term writes as
\begin{eqnarray*}
  \lefteqn{( f ( s, \mathcal{X}_{s} ) -f ( s, \mathcal{X}_{s_{-}} ) )
  \mathd \mathcal{N}^{i}_{s}}\\
   & = & \left( f \left( s, \mathcal{X}_{s_{-}}
  +G_{i}^{\tau_{\mathcal{N}_{s}}} \right) -f ( s, \mathcal{X}_{s_{-}} )
  \right)  \mathd \mathcal{N}^{i}_{s}\\
  & = & \left( X_{i}^{\tau_{\mathcal{N}_{s}}} f \right) ( s,
  \mathcal{X}_{s_{-}} ) \tau_{\mathcal{N}_{s}} \mathd
  \mathcal{N}^{i}_{s}\\
  & = & \frac{1}{2} \left\{ ( X_{i}^{2} f ) ( s, \mathcal{X}_{s_{-}} )
   + \tau_{\mathcal{N}_{s}}  ( X_{i}^{0}
  f ) ( s, \mathcal{X}_{s_{-}} ) \right\} \mathd
  \mathcal{N}^{i}_{s}
\end{eqnarray*}
where we introduced, for all $1 \leqslant i \leqslant m$,
\begin{eqnarray*}
  X_{i}^{0} & \assign & X_{i}^{+} +X_{i}^{-}\\
  X_{i}^{2} & \assign & X_{i}^{+} -X_{i}^{-} .
\end{eqnarray*}
Notice that, for any given $1 \leqslant i \leqslant m$, up to a normalisation
factor, the system of coordinate $( X_{i}^{2} ,X_{i}^{0} )$ is obtained thanks
to a {\tmem{rotation}} of $\pi /4$ of the canonical system of coordinate $(
X_{i}^{+} ,X_{i}^{-} )$. Finally,
\begin{eqnarray*}
 \lefteqn{ f ( t, \mathcal{X}_{t} ) - f ( 0, \mathcal{X}_{0} )}\\
  & = & \int_{0}^{t} ( \partial_{t} f ) ( s,
  \mathcal{X}_{s} ) \mathd s\\
  &&+ \frac{1}{2} \sum_{i=1}^{m} \int_{0}^{t} \left\{
  ( X_{i}^{2} f ) ( s, \mathcal{X}_{s_{-}} )  + 
  \tau_{\mathcal{N}_{s}}  ( X_{i}^{0} f ) ( s, \mathcal{X}_{s_{-}}
  ) \right\} \mathd \mathcal{N}^{i}_{s} \nonumber\\
  & = & \int_{0}^{t} \left\{ ( \partial_{t} f ) ( s, \mathcal{X}_{s} ) +
  \frac{\lambda}{2} ( \Delta_{x} f ) ( s, \mathcal{X}_{s} ) \right\} \mathd s+
  \nonumber\\
  &  & + \frac{1}{2} \sum_{i=1}^{m} \int_{0}^{t} ( X^{2}_{i} f ) ( s,
  \mathcal{X}_{s_{-}} )  \mathd ( \mathcal{N}^{i}_{s} - \lambda s
  ) + ( X^{0}_{i} f ) ( s, \mathcal{X}_{s_{-}} )  \mathd
  \mathcal{X}_{s}^{i} \nonumber\\
  &  & \tmop{having}   \tmop{set}   \mathd \mathcal{X}^{i}_{s} =
  \tau_{\mathcal{N}_{s}} \mathd \mathcal{N}^{i}_{s} \nonumber\\
  & = & \int_{0}^{t} \left\{ ( \partial_{t} f ) ( s, \mathcal{X}_{s} ) +
  \frac{\lambda}{2} ( \Delta_{x} f ) ( s, \mathcal{X}_{s} ) \right\} \mathd s+
  \int_{0}^{t} \left\langle \widehat{\tmmathbf{\mathd}} f ( s,
  \mathcal{X}_{s_{-}} ) , \mathd \widehat{\mathcal{W}}_{s} \right\rangle
  \nonumber\\
  & =: & \int_{0}^{t} \left\{ ( \partial_{t} f ) ( s, \mathcal{X}_{s} ) +
  \frac{\lambda}{2} ( \Delta_{x} f ) ( s, \mathcal{X}_{s} ) \right\} \mathd s+
  \left( I_{\widehat{\tmmathbf{\mathd}}_{x} f} \right)_{t}  . 
\end{eqnarray*}
It is easy to see that $  \mathd \mathcal{X} ^{i}_{s}$ is the stochastic
differential of a martingale. Here and in the sequel, we take $\lambda =2$.

\paragraph{Discrete It{\^o} integral}The stochastic integral above shows that
It{\^o} formula (\ref{eq: Ito for continuous processes}) for continuous
processes has a discrete counterpart involving stochastic integrals for jump
processes, namely we have the {\tmem{discrete}} It{\^o} integral
\begin{eqnarray*}
  \left( I_{\widehat{\tmmathbf{\mathd}}_{x} f} \right)_{t} & \assign &
  \frac{1}{2} \sum_{i=1}^{m} \int_{0}^{t} ( X^{2}_{i} f ) ( s,
  \mathcal{X}_{s_{-}} ) \  \mathd ( \mathcal{N}^{i}_{s} - \lambda s
  ) + ( X^{0}_{i} f ) ( s, \mathcal{X}_{s_{-}} ) \  \mathd
  \mathcal{X}_{s}
\end{eqnarray*}
This has a more intrinsic expression similar to the continuous It{\^o}
integral (\ref{eq: Ito for continuous processes}). If we regard the discrete
component $\mathbbm{G}_{x}$ as a ``discrete Riemannian'' manifold, then this
discrete It{\^o} integral involves discrete vectors (resp. $1$--forms) defined
on the {\tmem{augmented}} discrete tangent (resp. cotangent) space
$\hat{T}_{x} \mathbbm{G}_{x}$ (resp. $\hat{T}_{x}^{\ast} \mathbbm{G}_{x}$) of
dimension $2m$ defined as
\begin{eqnarray*}
  \hat{T}_{x} \mathbbm{G}_{x} & = & \tmop{span} \{ X_{1}^{+} ,X_{2}^{+} ,
  \ldots ,X_{1}^{-} ,X_{2}^{-} , \ldots \}\\
  & = & \tmop{span} \{ X_{1}^{2} ,X_{2}^{2} , \ldots ,X_{1}^{0} ,X_{2}^{0} ,
  \ldots \}\\
  \hat{T}_{x}^{\ast} \mathbbm{G}_{x} & = & \tmop{span} \{ ( X_{1}^{+} )^{\ast}
  , ( X_{2}^{+} )^{\ast} , \ldots , ( X_{1}^{-} )^{\ast} , ( X_{2}^{-}
  )^{\ast} , \ldots \}\\
  & = & \tmop{span} \{ ( X_{1}^{2} )^{\ast} , ( X_{2}^{2} )^{\ast} , \ldots ,
  ( X_{1}^{0} )^{\ast} , ( X_{2}^{0} )^{\ast} , \ldots \} .
\end{eqnarray*}
Let $\mathd \widehat{\mathcal{W}}_{s} \in \hat{T}_{\mathcal{X}_{s}}
\mathbbm{G}_{x}$ be the vector and $\widehat{\tmmathbf{\mathd}} f \in
\hat{T}_{\mathcal{X}_{s}}^{\ast} \mathbbm{G}_{x}$ be the $1$--form
respectively defined as:
\begin{eqnarray*}
  \mathd \widehat{\mathcal{W}}_{s} & = & \mathd ( \mathcal{N}^{1}_{s} -
  \lambda s )  X_{1}^{2} + \ldots + \mathd ( \mathcal{N}^{m}_{s} -
  \lambda s ) X_{m}^{2} + \mathd
  \mathcal{X}_{s}^{1} X_{1}^{0} + \ldots + \mathd
  \mathcal{X}_{s}^{m}  X_{m}^{0}\\
  \widehat{\tmmathbf{\mathd}}_{x} f & = & X_{1}^{2} f  ( X_{1}^{2}
  )^{\ast} + \ldots +X_{m}^{2} f ( X_{m}^{2} )^{\ast} +X_{1}^{0}
  f  ( X_{1}^{0} )^{\ast} + \ldots +X_{m}^{0} f  (
  X_{m}^{0} )^{\ast}
\end{eqnarray*}
We have with these notations
\begin{eqnarray*}
  \left( I_{\widehat{\tmmathbf{\mathd^{}}}_{x} f} \right)_{t} & \assign &
  \left\langle \widehat{\tmmathbf{\mathd}}_{x} f, \mathd
  \widehat{\mathcal{W}}_{s} \right\rangle_{\hat{T}_{x}^{\ast} \mathbbm{G}_{x}
  \times \hat{T}_{x} \mathbbm{G}_{x}}
\end{eqnarray*}
where the factor $1/2$ is included in the pairing $\langle \cdot , \cdot
\rangle_{\hat{T}_{x}^{\ast} \mathbbm{G}_{x} \times \hat{T}_{x}
\mathbbm{G}_{x}}$.

\paragraph{Semi-discrete stochastic integrals}Let finally $\mathcal{Z}_{t} = (
\mathcal{X}_{t} , \mathcal{Y}_{t} )$ be a semi-discrete random walk on the
cartesian product $\mathbbm{G}=\mathbbm{G}_{x} \times \mathbbm{G}_{y}$, where
$\mathcal{X}_{t}$ is the random walk above defined on $\mathbbm{G}_{x}$ with
generator $\Delta_{x}$ and where $\mathcal{Y}_{t}$ is the Brownian motion
defined on $\mathbbm{G}_{y}$ with generator $\Delta_{y}$. For $f \assign f (
t,z ) =f ( t,x,y )$ defined from $\mathbbm{R}^{+} \times \mathbbm{G}$ onto
$\mathbbm{C}$, we have easily the stochastic integral involving both discrete
and continuous parts:
\begin{eqnarray*}
  f ( t, \mathcal{Z}_{t} ) & = & \int_{0}^{t} \{ ( \partial_{t} f ) (
  s,\mathcal{Z}_{s} ) + ( \Delta_{z} f ) ( s,\mathcal{Z}_{s} ) \} \; \mathd s+
  \left( I_{\widehat{\tmmathbf{\mathd}}_{z} f} \right)_{t}
\end{eqnarray*}
where the {\tmem{semi-discrete}} It{\^o} integral writes as
\begin{eqnarray*}
  \left( I_{\widehat{\tmmathbf{\mathd}}_{z} f} \right)_{t} & \assign & \left(
  I_{\widehat{\tmmathbf{\mathd}}_{x} f} \right)_{t} + \left(
  I_{\tmmathbf{\mathd}_{y} f} \right)_{t}\\
  & \assign & \int_{0}^{t} \left\langle \widehat{\tmmathbf{\mathd}}_{x} f (
  s, \mathcal{Z}_{s_{-}} ) , \mathd \widehat{\mathcal{W}}_{s}
  \right\rangle_{\hat{T}_{\mathcal{X}_{s}}^{\ast} \mathbbm{G}_{x} \times
  \hat{T}_{\mathcal{X}_{s}} \mathbbm{G}_{x}} \\
  &&+ \int_{0}^{t} \left\langle
  \tmmathbf{\mathd}_{y} f ( s, \mathcal{Z}_{s_{-}} ) , \mathd \mathcal{Y}_{s}
  \right\rangle_{\hat{T}_{\mathcal{Y}_{s}}^{\ast} \mathbbm{G}_{y} \times
  \hat{T}_{\mathcal{Y}_{s}} \mathbbm{G}_{y}}
\end{eqnarray*}

\subsection{Martingale transforms and quadratic covariations}

\mbox{}\\

\paragraph{Martingale transforms} We are interested in martingale transforms
allowing us to represent second order Riesz transforms. Let $f ( t,z )$ be a
solution to the heat equation $\partial_{t} - \Delta_{z} =0$. Fix $T>0$ and
$\mathcal{Z}_{0} \in \mathbbm{G}$. Then define
\[ \forall 0 \leqslant t \leqslant T, \hspace{1em} M_{t}^{f,T,\mathcal{Z}_{0}}
   =f ( T-t,\mathcal{Z}_{t} ) . \]
This is a martingale since $f ( T-t )$ solves the backward heat equation
$\partial_{t} + \Delta_{z} =0$, and we have in terms of stochastic integrals
\[ M_{t}^{f,T,\mathcal{Z}_{0}} =f ( T-t,\mathcal{Z}_{t} ) =f (
   T,\mathcal{Z}_{0} ) + \int_{0}^{t} \left\langle
   \widehat{\tmmathbf{\mathd}}_{z} f ( T-s,\mathcal{Z}_{s_{-}} ) , \mathd
   \mathcal{Z}_{s} \right\rangle \]
Given $\tmmathbf{A}_{\alpha}$ the $\mathbbm{C}^{( 2m+n ) \times ( 2m+n )}$
matrix defined earlier, we note $M_{t}^{\alpha ,f,T,\mathcal{Z}_{0}}$ the
martingale transform $\tmmathbf{A}_{\alpha} \ast M_{t}^{f,T,\mathcal{Z}_{0}}$
defined as
\begin{eqnarray*}
  M_{t}^{\alpha ,f,T,\mathcal{Z}_{0}} & \assign & f ( T,\mathcal{Z}_{0} ) +
  \int_{0}^{t} \left( \tmmathbf{A}_{\alpha} \widehat{\tmmathbf{\nabla}}_{z} f
  ( s,\mathcal{Z}_{s_{-}} ) , \mathd \mathcal{Z}_{s} \right)\\
  & = & f ( T,\mathcal{Z}_{0} ) + \int_{0}^{t} \left\langle
  \widehat{\tmmathbf{\mathd}}_{z} f ( T-s,\mathcal{Z}_{s_{-}} )
  \tmmathbf{A}_{\alpha}^{\ast} , \mathd \mathcal{Z}_{s} \right\rangle
\end{eqnarray*}
where the first integral involves the $L^{2}$ scalar product on $\hat{T}_{z}
\mathbbm{G} \times \hat{T}_{z} \mathbbm{G}$ and the second integral involves
the duality $\hat{T}_{z}^{\ast} \mathbbm{G} \times \hat{T}_{z} \mathbbm{G}$.
In differential form:
\begin{eqnarray*}
  \lefteqn{\mathd M_{t}^{\alpha ,f,T,\mathcal{Z}_{0}} }\\
  & = & \left(
  \tmmathbf{A}_{\alpha} \widehat{\tmmathbf{\nabla}}_{z} f (
  s,\mathcal{Z}_{s_{-}} ) , \mathd \mathcal{Z}_{s} \right)\\
  & = & \sum_{i=1}^{m} \alpha_{i}^{x}  \left\{ ( X^{2}_{i} f ) (
  T-t,\mathcal{Z}_{t_{-}} ) \; \mathd ( \mathcal{N}^{i}_{t} - \lambda t ) + (
  X^{0}_{i} f ) ( t,\mathcal{Z}_{t_{-}} ) \; \mathd \mathcal{X}_{t}^{i}
  \right\}\\
  &  & + \sum_{j=1}^{n} \alpha_{j,k}^{y} \; ( X_{j} f ) (
  T-t,\mathcal{Z}_{t_{-}} ) \; \mathd \mathcal{Y}_{t}^{k}
\end{eqnarray*}

\paragraph{Quadratic covariation and subordination}We have the quadratic
covariations (see {\tmname{Protter }}{\cite{Pro2005a}},
{\tmname{Dellacherie}}--{\tmname{Meyer}} {\cite{DelMey1982a}}, or
{\tmname{Privault}} {\cite{Pri2009a,Pri2014a}})
\begin{eqnarray*}
  \mathd [ \mathcal{N}^{i} - \lambda t, \mathcal{N}^{i} - \lambda t ]_{t} & =
  & \mathd \mathcal{N}_{t}^{i}\\
  \mathd [ \mathcal{N}^{i} - \lambda t, \mathcal{X}^{i} ]_{t} & = &
  \tau_{\mathcal{N}_{t}} \  \mathd \mathcal{N}_{t}^{i}\\
  \mathd [ \mathcal{X}^{i} , \mathcal{X}^{i} ]_{t} & = & \mathd
  \mathcal{N}_{t}^{i}\\
  \mathd [ \mathcal{Y}^{j} , \mathcal{Y}^{j} ]_{t} & = & \mathd t,
\end{eqnarray*}
the other quadratic covariations being zero. For any two martingales
$M_{t}^{f}$ and $M_{t}^{g}$ defined as above thanks to their respective heat
extensions $P_{t} f$ et $P_{t} g$, we have the quadratic covariations
\begin{eqnarray*}
  \lefteqn{\mathd [ M^{f} ,M^{g} ]_{t} }\\
  & = & \sum_{i=1}^{m} ( X^{2}_{i} f ) ( T-t,
  \mathcal{Z}_{t_{-}} ) \; ( X^{2}_{i} g ) ( T-t, \mathcal{Z}_{t_{-}} )
   \mathd [ \mathcal{N}^{i} - \lambda t, \mathcal{N}^{i} - \lambda
  t ]_{t}\\
  &  & + \sum_{i=1}^{m} ( X^{0}_{i} f ) ( T-t,
  \mathcal{Z}_{t_{-}} ) \; ( X^{0}_{i} g ) ( T-t, \mathcal{Z}_{t_{-}} )
  \hspace{1em} \mathd [ \mathcal{X}^{i} , \mathcal{X}^{i}_{} ]_{t}\\
  &  &  +  \sum_{i=1}^{m} ( X^{2}_{i} f ) ( T-t,
  \mathcal{Z}_{t_{-}} ) \; ( X^{0}_{i} g ) ( T-t, \mathcal{Z}_{t_{-}} )
  \mathd [ \mathcal{N}^{i} - \lambda t, \mathcal{X}^{i}_{}
  ]_{t}\\
  &  & +  \sum_{i=1}^{m} ( X^{0}_{i} f ) ( T-t,
  \mathcal{Z}_{t_{-}} ) \; ( X^{2}_{i} g ) ( T-t, \mathcal{Z}_{t_{-}} )
   \mathd [ \mathcal{X}^{i} , \mathcal{N}^{i}_{} - \lambda t
  ]_{t}\\
  &  & +  \sum_{j=1}^{n} ( X_{j} f ) ( T-t,
  \mathcal{Z}_{t_{-}} ) \; ( X_{j} g ) ( T-t, \mathcal{Z}_{t_{-}} )
   \mathd [ \mathcal{Y}^{j} , \mathcal{Y}^{j} ]_{t}\\
  & = & \sum_{i=1}^{m} \sum_{\tau = \pm} ( X^{\tau}_{i} f ) \; (
  X^{\tau}_{i} g ) ( T-t, \mathcal{Z}_{t_{-}} ) \mathbbm{1} (
  \tau_{\mathcal{N}_{t}} = \tau 1 )   \mathd \mathcal{N}^{i}_{t}\\
  &  &  + \left( \tmmathbf{\nabla}_{y} f, \tmmathbf{\nabla}_{y} g
  \right) ( T-t, \mathcal{Z}_{t_{-}} )  \mathd t
\end{eqnarray*}
so that finally
\begin{eqnarray} \label{eq:quadratic_covariation_of_f_and_g}
  \mathd [ M^{f} ,M^{g} ]_{t} &=& \sum_{i=1}^{m} \sum_{\tau = \pm} 
  ( X^{\tau}_{i} f ) \; ( X^{\tau}_{i} g ) ( T-t, \mathcal{Z}_{t_{-}} )
   \mathbbm{1} ( \tau_{\mathcal{N}_{t}} = \tau 1 )   \mathd
  \mathcal{N}^{i}_{t} \\
  &&+ \left( \tmmathbf{\nabla}_{y} f, \tmmathbf{\nabla}_{y}
  g \right) ( T-t, \mathcal{Z}_{t_{-}} )  \mathd t \nonumber
\end{eqnarray}

\paragraph{Differential subordination}Following {\tmname{Wang}}
{\cite{Wan1995a}}, given two adapted c{\`a}dl{\`a}g Hilbert space valued
martingales $X_{t}$ and $Y_{t}$, we say that $Y_{t}$ is differentially
subordinate by quadratic variation to $X_{t}$ if $| Y_{0} |_{\mathbbm{H}}
\leqslant | X_{0} |_{\mathbbm{H}}$ and $[ Y,Y ]_{t} - [ X,X ]_{t}$ is
nondecreasing nonnegative for all $t$. In our case, we observe that
\begin{eqnarray*}
  \mathd [ M^{\alpha ,f} ,M^{\alpha ,f} ]_{t} & = & \sum_{i=1}^{m} |
  \alpha_{i}^{x} |^{2} \; \left\{ ( X^{+}_{i} f )^{2} ( T-t,
  \mathcal{Z}_{t_{-}} )  \mathbbm{1} ( \tau_{\mathcal{N}_{t}} =1 )
  \right.\\
  &  & +  \left. ( X^{-}_{i} f )^{2} ( T-t,
  \mathcal{Z}_{t_{-}} )  \mathbbm{1} ( \tau_{\mathcal{N}_{t}} =-1
  ) \right\}  \mathd \mathcal{N}^{i}_{t}\\
  &  &  + \left( \tmmathbf{A}^{y}_{\alpha}  
  \tmmathbf{\nabla}_{y} f,\tmmathbf{A}^{y}_{\alpha}   \tmmathbf{\nabla}_{y} f
  \right) ( T-t, \mathcal{Z}_{t_{-}} ) \mathd t.
\end{eqnarray*}
Hence
\begin{equation}
  \mathd [ M^{\alpha ,f} ,M^{\alpha ,f} ]_{t} \leqslant \|
  \tmmathbf{A}_{\alpha} \|^{2}_{2}  \  \mathd [ M^{f} ,M^{f} ]_{t}
  \label{eq: differential subordination} .
\end{equation}
This means that $M^{\alpha ,f}_{t}$ is differentially subordinate to $\|
\tmmathbf{A}_{\alpha} \|_{2} M_{t}^{f}$.

\section{Proofs of the main results}\label{S: proofs of the main results}

\subsection{Proof of Theorem \ref{T: a la Gundy-Varopoulos}}To establish the
estimate itself, we adapt the well-known connection between martingale
transforms and classical singular integral operators, through the use of
projection operators. We refer to {\tmname{Gundy}}--{\tmname{Varopoulos}}
{\cite{GunVar1979}} as well as {\tmname{Ba{\~n}uelos}} {\cite{Ban1986a}} and
{\tmname{Ba{\~n}uelos}}--{\tmname{Baudoin}} {\cite{BanBau2013}}. Following the
same strategy, the random trajectories we use $( \mathcal{B}_{t} )_{-T \leqslant t
\leqslant 0}$ defined on the band $[ -T,0 ] \times \mathbbm{G}$ by
\[ \mathcal{B}_{t} \assign ( -t, \mathcal{Z}_{t} ) , \hspace{1em}
   \mathcal{B}_{-T} = ( T, \mathcal{Z}_{-T} ) , \hspace{1em} -T \leqslant t
   \leqslant 0, \hspace{1em} \mathcal{Z}_{-T} \in \mathbbm{G} \]
are replaced by random trajectories $( \mathcal{B}_{t} )_{- \infty \leqslant t
\leqslant 0}$ defined on the upper half space $\mathbbm{R}^{+} \times
\mathbbm{G}$ after exhaustion of the upper half space.
\[ \mathcal{B}_{t} \assign ( -t, \mathcal{Z}_{t} ) , \hspace{1em}
   \mathcal{B}_{- \infty} = ( \infty , \mathcal{Z}_{- \infty} ) , \hspace{1em}
   - \infty \leqslant t \leqslant 0, \hspace{1em} \mathcal{Z}_{- \infty} \in
   \mathbbm{G}. \]
The latter are therefore trajectories starting at time $t=- \infty$ from a
chosen point $\mathcal{Z}_{- \infty} \in \mathbbm{G}$, and stopping at time
$t=0$, when hitting the bottom boundary $\mathbbm{G}$ of the upper half space.
If $f ( t ) =P_{t} f$ is as in the previous section, then $M_{t}^{f} =f (
\mathcal{B}_{t} )$ for $- \infty \leqslant t \leqslant 0$ is a martingale and
$M_{t}^{\alpha ,f}$ its martingale transform as defined previously. Let the
projection operator $\mathcal{T}^{\alpha}$ be defined as the following
conditional expectation of the martingale transform $M^{\alpha ,f}_{t}$ of the
stochastic integral $M_{t}^{f}$:
\[ \forall z \in \mathbbm{G}, \hspace{1em} ( \mathcal{T}^{\alpha} f ) ( z )
   \assign \mathbbm{E} \left( M_{0}^{\alpha ,f} \hspace{.5em} \vert
   \hspace{.5em} \mathcal{Z}_{0} =z \right) . \]
Following {\tmname{Gundy}}-{\tmname{Varopoulos}} {\cite{GunVar1979}} (see also
{\cite{BanBau2013}}) this operator is the second order Riesz transform we are
interested in. Indeed, recalling the expression (\ref{eq:quadratic_covariation_of_f_and_g}) 
of the quadratic covariations of two martingale
increments, it is not difficult to calculate

{\hspace*{\fill}}
\begin{eqnarray*}
  \forall g, \hspace{1em} ( \mathcal{T}^{\alpha} f,g ) & = &
  \int_{\mathbbm{G}} ( \mathcal{T}^{\alpha} f ) ( z ) g ( z ) \; \mathd
  \mu_{z} ( z )\\
  & = & 2 \int^{\infty}_{0} \left( \tmmathbf{A}_{\alpha}
  \widehat{\tmmathbf{\nabla}}_{z} P_{t} f, \widehat{\tmmathbf{\nabla}}_{z}
  P_{t} g \right)_{L^{2} ( \hat{T} \mathbbm{G} )} \; \mathd  t.
\end{eqnarray*}
This means thanks to Lemma \ref{L: representation formula} that
$\mathcal{T}^{\alpha} =R_{\alpha}^{2}$. This concludes the proof of Theorem
\ref{T: a la Gundy-Varopoulos}.{\hfill}$\Box$

\subsection{Proof of Theorem \ref{T: p minus 1 estimate}}Recall that the
subordination estimate (\ref{eq: differential subordination}) shows that the
martingale transform $Y_{t} \assign M_{t}^{\alpha}$ is differentially
subordinate to the martingale $X_{t} \assign \| \tmmathbf{A}_{\alpha} \|_{2}
\hspace{.1em} M_{t}^{f}$. Following this result of {\tmname{Wang}}
{\cite{Wan1995a}}:

\begin{theorem}
  \label{T: Wang's result}{\dueto{Wang, 1995}}Let $X_{t}$ and $Y_{t}$ be two
  adapted c{\`a}dl{\`a}g Hilbert--valued martingales such that $Y_{t}$ is
  differentially subordinate by quadratic covariation to $X_{t}$. For $1<p<
  \infty$,
  \[ \| Y_{t} \|_{p} \leqslant ( p^{\ast} -1 ) \| X_{t} \|_{p} \]
  and the constant $p^{\ast} -1$ is best possible. Strict inequality holds
  when $0< \| X \|_{p} < \infty$ and $p \neq 2$,
\end{theorem}

we get immediately

\begin{lemma}
  \label{L: subordination of the martingale transform}Let $M_{t}^{f}$ and
  $M_{t}^{\alpha ,f}$ as defined above. We have
  \[ \forall t, \hspace{1em} \| M_{t}^{\alpha ,f} \|_{p} \leqslant \|
     \tmmathbf{A}_{\alpha} \|_{2} \hspace{.1em} ( p^{\ast} -1 ) \hspace{.1em} \|
     M_{t}^{f} \|_{p} . \]
\end{lemma}

Finally, the operator $\mathcal{T}^{\alpha}$ being a conditional expectation
of $M_{t}^{\alpha ,f}$, this proves the estimate $\| \mathcal{T}^{\alpha}
\|_{p} \leqslant \| \tmmathbf{A}_{\alpha} \|_{2} \hspace{.1em} ( p^{\ast} -1
)$.

\paragraph{Sharpness} The sharpness in Lie groups with at least two infinite directions is inherited from the continuous case, where it is seen that two continuous directions are enough for sharpness and optimality of the estimate is seen when considering the operator $R^2_1-R^2_2$. If one or both continuous directions are replaced by $\mathbb{Z}$ just consider the isomorphic groups $( t\mathbb{Z} )^{N}$
for $0<t \leqslant 1$ in conjunction with the Lax equivalence theorem
{\cite{LaxRic1956a}}. By the same argument, sharpness for a {\tmem{uniform}}
estimate in $m$ for the cyclic case $( \mathbb{Z}/m\mathbb{Z} )^{N}$ is
inherited from that on the torus $\mathbb{T}^{N}$.

\subsection{Proof of Theorem \ref{T: Choi constant estimate}}The proof of
Theorem \ref{T: Choi constant estimate} follows exactly the same procedure.
Recall Choi's result {\cite{Cho1992a}} for discrete martingales.

\begin{theorem}
  {\dueto{Choi, 1992}}Let $( \Omega , ( \mathcal{F} )_{n \in \mathbbm{N}}
  ,\mathbbm{P} )$ a probability space and $X_{n}$ an adapted real valued
  martingale. Let $( \alpha_{n} )_{n \in \mathbbm{N}}$ be a predictable
  sequence taking values in $[ 0,1 ]$. Let $Y \assign \alpha \ast X$ be the
  martingale transform of $X$ defined for almost all $\omega \in \Omega$ as
  \[ Y_{0} ( \omega ) =X_{0} ( \omega ) , \hspace{1em} \tmop{and} \hspace{1em}
     ( Y_{n+1} -Y_{n} ) ( \omega ) = \alpha_{n} \; ( X_{n+1} -X_{n} ) ( \omega
     ) . \]
  Then there exists a constant $\mathfrak{C}_{p}$ depending only on $p$ such
  that $\| Y \|_{p} \leqslant \mathfrak{C}_{p} \| X \|_{p}$ and the estimate
  is best possible.
\end{theorem}

The previous result from Choi is only for discrete martingales. For
continuous-in-time martingales, we invoke Theorem 1.6 from the result of
{\tmname{Ba{\~n}uelos}} and {\tmname{Osekowski}} {\cite{BanOse2012a}}, namely

\begin{lemma}
  \label{L: Choi Lp norms of martingales}{\dueto{Banuelos--Osekowski,
  2012}}Let $X_{t}$ and $Y_{t}$ be two real-valued martingales satisfying
  \[ \mathd \left[ Y- \frac{a+b}{2} X,Y- \frac{a+b}{2} X \right]_{t} \leqslant
     \mathd \left[ \frac{b-a}{2} X, \frac{b-a}{2} X \right]_{t} \]
  for all $t \geqslant 0$. Then for all $1<p< \infty$, we have $\| Y \|_{p}
  \leqslant \mathfrak{C}_{p} \| X \|_{p}$.
\end{lemma}

The result is now a corollary of Lemma \ref{L: Choi Lp norms of martingales}
above with $X_{t} =M_{t}^{f}$ and $Y_{t} =M_{t}^{\alpha ,f}$. It is not
difficult to prove that the difference of quadratic variations above writes in
terms of a jump part and a continuous part as
 \begin{eqnarray*}
     \lefteqn{ \left[ Y- \frac{a+b}{2} X,Y- \frac{a+b}{2} X \right]_{t} - \mathd
     \left[ \frac{b-a}{2} X, \frac{b-a}{2} X \right]_{t} }\\
     &=& \sum_{i=1}^{m} \sum_{\pm} ( \alpha_{i}^{x} -a ) (
     \alpha_{i}^{x} -b ) \hspace{.1em} ( X^{\pm}_{i} f )^{2} ( \mathcal{B}_{t}
     ) \hspace{.1em} \mathbbm{1} ( \tau_{N_{t}} = \pm 1 ) \hspace{.1em} \mathd
     \mathcal{N}_{t}^{i}\\
     &&+ \left\langle \left( \tmmathbf{A}_{\alpha}^{y} -a
     \tmmathbf{I} \right) \left( \tmmathbf{A}_{\alpha}^{y} -b \tmmathbf{I}
     \right) \hspace{.1em} \tmmathbf{\nabla}_{y} f ( \mathcal{B}_{t} )
     ,  \tmmathbf{\nabla}_{y} f ( \mathcal{B}_{t} )
     \right\rangle \hspace{.1em} \mathd t,
   \end{eqnarray*} 
which is nonpositive since we assumed precisely $a \tmmathbf{I} \leqslant
\tmmathbf{A}_{\alpha} \leqslant b \tmmathbf{I}$. This proves Theorem \ref{T:
Choi constant estimate}.{\hfill}$\Box$

\end{document}